\documentclass[a4paper,12pt]{amsart}

\pdfoutput=1
\usepackage{setspace}
\usepackage[left=22mm,head=30mm,bottom=30mm,foot=20mm,right=22mm]{geometry}
\usepackage{amsmath,amsthm,amsfonts,amssymb,mathtools,mathrsfs,url,bm,enumitem,stackengine}
\usepackage{newtxtext}
\usepackage[varvw]{newtxmath}
\allowdisplaybreaks
\usepackage{hyperref}
\hypersetup{
	colorlinks=true,
	linkcolor=blue,
	citecolor=red,
}
\usepackage{appendix}

\mathtoolsset{showonlyrefs=true}
\theoremstyle{definition}

\newtheorem{remark}{Remark}[section]

\newtheorem{conjecture}{Conjecture}[section]

\theoremstyle{plain}

\newtheorem{theorem}{Theorem}[section]

\def\R{\mathbb{R}}

\renewcommand{\phi}{\varphi}
\newcommand\restr[2]{{
		\left.\kern-\nulldelimiterspace 
		#1 
		\vphantom{\big|} 
		\right|_{#2} 
}}

\let\phi\varphi
\let\epsilon\varepsilon

\usepackage{tikz-cd}
\tikzset{
	symbol/.style={
		draw=none,
		every to/.append style={
			edge node={node [sloped, allow upside down, auto=false]{$#1$}}}
	}
}

\def\R{\mathbb R}

\def\proscal3#1#2{<\!\!#1, #2\!\!>_{_{{\hskip-4pt\R^{3}}}}}

\def\vu{ u }
\def\vn{ \nabla }

\def\l2{L^2(\R^{n})}
\def\L2{L^2(\R^{2n})}
\def\supp{\operatorname{Supp}}

\def\mat22#1#2#3#4{\begin{pmatrix}#1&#2\\ #3&#4\end{pmatrix}}

\begin{document}

	\title[Liouville Theorems for SNS via the Radial Velocity Component]{Liouville Theorems for Stationary Navier-Stokes Equations via the Radial Velocity Component}

	

	\author{Gast\'{o}n Vergara-Hermosilla}
	\address{School of Mathematics, Harbin Institute of Technology, Xidazhi Street, Harbin 150001,
Heilongjiang, People’s Republic of China}
\email{gaston.v-h@outlook.com (corresponding author)}	
	
	\maketitle

\begin{abstract}
We study Liouville-type results for the stationary Navier--Stokes equations in $\mathbb{R}^3$. We prove that any $\dot{H}^1(\mathbb{R}^3)$ solution is trivial under an integrability condition imposed only on the radial component of the velocity, namely $u_\rho(x) \in L^p(\mathbb{R}^3)$ with $3/2 < p \leq 3$.  . Our analysis reveals that the rigidity of the stationary Navier--Stokes system can be driven by localized and  radial  integrability properties, rather than uniform global conditions.
\end{abstract}

 \section{Introduction and main results}
 
 	In this note we consider the stationary Navier-Stokes equations 
	\begin{equation}\label{SNS}
			-\Delta u+  ( u\cdot \nabla ) u +\nabla \pi  =0,
			\quad \nabla \cdot u=0, 
		\quad 
		\text{in }\R^3,
	\end{equation} 
 where $\vu $ stands the velocity vector field and $\pi  $ denotes the  pressure. The construction of solutions $(\vu, \pi)$ for this equation in the spaces $\dot{H}^1(\mathbb{R}^3) \times \dot{H}^{\frac{1}{2}}(\R^3) $  is classical (see \cite[Theorem 16.2]{lemarie2016navier}), however  uniqueness of solutions is  until now a challenging and open problem. This motivates the following conjecture, which has been originally proposed by G. Galdi in \cite[Remark X.9.4 and Theorem X.9.5]{galdi2011introduction} and by G. Seregin in \cite{Ser2016}):
\begin{conjecture}
Show that any solution $\vu$ of (\ref{SNS}) fulfilling the conditions
\begin{equation}\label{Conjecture}
\vu \in \dot{H}^1(\R^3) \qquad\mbox{and}\qquad \vu(x)\to 0 \mbox{ as } |x|\to +\infty,
\end{equation}
is identically equal to zero.
\end{conjecture}
Note that the Liouville problem in the two-dimensional case was established by Gilbarg and Weinberger in \cite{GW78}, while the four-dimensional case was proved by G. Galdi in \cite{galdi2011introduction}. The three-dimensional case, however, remains widely open and has attracted significant interest within the mathematical fluid mechanics community.
On the other hand, by Sobolev embedding, any solution
 $u \in \dot{H}^1(\mathbb{R}^3)$ also belongs to $L^6(\mathbb{R}^3)$ which already implies a certain decay as $x\to \infty$. However, this information alone does not appear to be sufficient to deduce the triviality of the solution. Over the years, several partial results related to Conjecture \ref{Conjecture} have been obtained, showing that additional integrability or structural assumptions ensure that $u=0$. In this direction  G.~Galdi  proved in \cite{galdi2011introduction} that $u \in L^{9/2}(\mathbb{R}^3)$ in fact implies  that $u \equiv 0$. This result was later refined by D.~Chae and J.~Wolf in \cite{ChaeWolf}, who obtained a logarithmic improvement of the $9/2$ condition.    N.~Lerner observed in \cite{Lerner26} that the global $L^{9/2}$ assumption can be relaxed by separating low and high frequencies. More precisely, he showed that it suffices to assume $u_{[0]} \in L^{9/2}(\R^3)$, with  $u_{[0]}$ denoting a projection of $u$ onto the subspace of vector fields whose Fourier support contains a neighborhood of the origin. 
 For more results about Liouville type theorems, we refer to the reader to Chapter 16 in the book of P.G. Lemari\'e-Rieusset \cite{lemarie2016navier} and the references therein.\\ 

To the best of the author’s knowledge, most Liouville-type results in the literature are isotropic, i.e., invariant under rotations of the coordinate system, while anisotropic results remain scarce. A seminal work in this  direction is due to D. Chae \cite{Chae23}, who proved uniqueness by considering sufficient conditions depending on the direction of coordinates and each component of the velocity on mixed norm Lebesgue spaces. This results was extended by Z. Zhang and Q. Zu in \cite{ZZ25}. 
In this paper, we study new sufficient conditions for Liouville-type theorems that depend explicitly only on  the radial component of the velocity field, thereby offering a different perspective on this problem.

 \begin{theorem}\label{thm1} 
 Let $\vu  \in \dot{H}^1(\mathbb{R}^3) $ be a weak solution of (\ref{SNS})
 and let $u_\rho(x)$ be the radial component of such solution in spherical coordinates. 
 Suppose in addition that
 $u_\rho \in L^{p}(\mathbb{R}^3)$ with $3/2<p\leq 3$ .
 Then, $u $ is identically equal to zero.
\end{theorem}

\begin{remark}
This result establishes the uniqueness of steady solutions under a weaker additional integrability assumption compared to known isotropic and anisotropic Liouville-type results for the stationary Navier--Stokes equations. 
More precisely, Theorem \ref{thm1} imposes sufficient conditions only on the radial component of the velocity field, without requiring any restrictions on the remaining two components.
\end{remark}

\begin{remark}
An interesting case in the Liouville theorem of D. Chae in \cite{Chae23} corresponds to the assumption (where 
$\tilde x_1 = (x_2, x_3), \ \tilde x_2= (x_1, x_3), \ \tilde x_3 = (x_1, x_2)$) 
\[
u_i \in \big(L_{x_i}^{\frac{3}{2}} \cap L_{x_i}^{6}\big) L_{\tilde x_i}^{6}(\mathbb{R} \times \mathbb{R}^2), 
\quad \text{for each cartesian coordinate indexed with } i = 1,2,3.
\]
This condition reflects an anisotropic decay, combining  mild integrability 
in planar directions with stronger decay in the orthogonal direction, and is sufficient 
to ensure the triviality of the solution.
By contrast, Theorem \ref{thm1} relies on the assumption 
$u_\rho \in L^{p}(\mathbb{R}^3)$ with $3/2 < p \leq 3$. In particular, this involves 
a range of integrability exponents, and only on a single component 
of the velocity field. 
\end{remark}

\begin{remark}
A particular case of our result corresponds to the assumption 
$u_\rho \in L^{3}(\mathbb{R}^3)$. At a heuristic level, this can be compared with the classical 
a priori condition $u \in L^{9/2}(\mathbb{R}^3)$ due to G.~Galdi 
\cite{galdi2011introduction};   
the latter integrability exponent naturally arises by applying Hölder's inequality 
to an integral term involving a cubic power of $|u|$, 
which leads to the balance condition
\[
\frac{2}{3} = \frac{3}{p},
\]
whose unique solution is $p = \frac{9}{2}$.
In contrast, our approach relies on a different distribution of the analogous integral term, 
involving two factors of the full velocity field $u$ and one factor of its radial 
component $u_\rho$. This leads, via Hölder's inequality, to the relation
\[
\frac{2}{3} = \frac{2}{p_1} + \frac{1}{p_2}.
\]
A natural and convenient choice of exponents is given by $p_1 = 6$ and $p_2 = 3$. 
This selection is consistent with the Sobolev embedding 
$\dot H^1(\mathbb{R}^3) \hookrightarrow L^6(\mathbb{R}^3)$, 
which provides control of the velocity field in $L^6$, while the radial 
component is handled in $L^3$. 
This argument also plays a role in the proof of our theorem, specifically in the control of the decay of terms involving the head pressure 
$Q = \pi + \frac{|u|^2}{2}$ (see \eqref{forheuristic1}--\eqref{forheuristic2} for further details).
\end{remark}
 
\section{Proof of main theorems}\label{sec:3} 

We begin this section by introducing some notations used throughout this paper; given $R>1$, we consider 
\[
C( R) := \left\{ x \in \mathbb{R}^3 \ \big | \ \frac R2 < |x| < R \right\}.
\]
\begin{proof}[Proof of Theorem \ref{thm1}]
To begin note that, since $\vu  \in \dot{H}^1(\mathbb{R}^3) $, we have $\vu \in L^6(\mathbb{R}^3) $. Then, considering that $L^6(\mathbb{R}^3) \subset L^6_{\text{loc} }(\mathbb{R}^3)  \subset L^3_{\text{loc} } (\mathbb{R}^3)$, 
thus, by Theorem X.1.1 in \cite{galdi2011introduction} we conclude that $\vu$ and $P$  are in fact a couple of regular functions. 
{\color{black}
To continue, let $\varphi\in \mathcal{C}^\infty([0,\infty))$ be a non-increasing cut-off
function such that $0\le \varphi\le 1$, $\varphi(r)=1$ if $r\leq \frac 1 2$, and
$\varphi(r)=0$ if $r\geq 1$. We define the radial, non-increasing cut-off
function $\psi\in \mathcal{C}^\infty_0(\mathbb{R}^3)$ by $\psi(x)=\varphi(|x|)$.
Given $R>1$, we set
$ 
\psi_R(x) = \psi\!\left(\frac{x}{R}\right) = \varphi\!\left(\frac{|x|}{R}\right),
$ 
so that $\psi_R$ is radial and non-increasing (as a function of $|x|$).
}
{ \color{black}
Now, note that, since $u\in \dot H^{1}(\mathbb{R}^3)$, the Sobolev
embedding gives $u\in L^{6}(\mathbb{R}^3)$, and therefore $u_iu_j\in L^{3}(\mathbb{R}^3)$. Then, taking the divergence of the momentum equation and using the divergence free condition, we get 
\begin{equation}\label{ellipeq}
-\Delta \pi=\partial_i\partial_j(u_iu_j)=:f .
\end{equation}
Now, since \eqref{ellipeq} determines $p$ only up to an additive harmonic function, we fix the
pressure through the Calder\'on--Zygmund representation
\[
\pi:=R_iR_j(u_iu_j),
\]
where $R_i$ stands for the  $i$-th Riesz transform on $\mathbb{R}^3$. Thusm By the Calder\'on--Zygmund
theorem, we obtain the estimate
\[
\|\pi\|_{L^{3}(\mathbb{R}^3)}\le C\,\|u\|_{L^{6}(\mathbb{R}^3)}^2<+\infty .
\]
We claim that this is the only solution of \eqref{ellipeq} in $L^{3}(\mathbb{R}^3)$. Indeed, if
$\tilde \pi\in L^{3}(\mathbb{R}^3)$ is another solution, then $h:=\pi-\tilde \pi$ is harmonic on
$\mathbb{R}^3$ and belongs to $L^{3}(\mathbb{R}^3)$. Thus, considering the mean value property and H\"older's
inequality, for every $x\in\mathbb{R}^3$ and $\rho>0$,
\[
|h(x)|\le \frac{1}{|B_\rho|}\int_{B_\rho(x)}|h|\,dy
\le |B_\rho|^{-\frac{2}{6}}\,\|h\|_{L^{3}(\mathbb{R}^3)}\xrightarrow[\rho\to\infty]{}0 ,
\]
thus, we get $h\equiv 0$. Throughout the paper, $\pi$ always denotes this normalized pressure. 
}
 Then, by testing the 3d stationary Navier-Stokes equations \eqref{SNS} with $\psi_R \vu$ and using the fact that $\supp(\psi_R \vu)\subseteq B(R):=B(0,R)$ we obtain \footnote{Note that, since the couple $\vu$ and $\pi$ are regular, the terms involved in the equality   are well-defined.}
\begin{equation}\label{eq. 3}
\int_{B(R)}-\Delta \vu \cdot\left(\psi_R \vu\right)+(\vu \cdot \vn) \vu \cdot\left(\psi_R \vu\right)+\vn \pi \cdot\left(\psi_R \vu\right) d x=0.
\end{equation}
Then, by using the divergence-free condition $\nabla \cdot \vu=0$ and integration by parts, we obtain 
the identity 
\begin{equation*}
\int_{B(R)} \psi_R|\vn \otimes \vu|^2 dx=\int_{B(R)} \Delta \psi_R \frac{|\vu|^2}{2} dx+
\int_{B(R)} \vn \psi_R \cdot\left[\left(\pi+\frac{|\vu|^2}{2}\right) \vu\right]dx.
\end{equation*}
Then, 
since  $\psi_R(x)= 1$ if $|x|< \frac R 2$, we can write
\begin{equation}\label{ineqmother}
\int_{B(R/2)}|\vn \otimes \vu|^2 dx 
 \leq  
\int_{B(R)} \Delta \psi_R \frac{|\vu|^2}{2} dx+ \int_{B(R)} \vn \psi_R \cdot\left[\left(\pi+\frac{|\vu|^2}{2}\right) \vu\right]dx 
=: 
I_1(R)+I_2(R)
. 
\end{equation}
In  the following we will prove 
\begin{equation}\label{LimitsAlphaBeta}
\displaystyle \lim_{R\to +\infty}  |I_1(R)|=\lim_{R\to +\infty}  |I_2(R)|=0.
\end{equation}
{\bf Limit for $I_1(R)$}. For studying the term $I_1(R)$, since the support of $\Delta\psi_R$ is contained in $C(R)$,  the H\"older inequality with 
\begin{equation}\label{variable_exponents_1}
1=\frac{2}{6}+\frac{2}{3}
\end{equation} yields  
\begin{equation}\label{202618aveq32_T1}
|I_1(R)|\leq C
\|\Delta\psi_R\|_{L^{\frac 3 2}(C( R) )}
\|\vu\|_{L^{6}( C( R)  )}^2.
\end{equation}
 In order to control the quantity $\|\Delta \psi_R \|_{L^{  \frac 3 2 }(\mathcal{C}( R) )}$ above,
we write
\begin{equation}\label{EstimationNormeLinfiniTheta_R1}
\|\Delta \psi_R \|_{L^{ \frac 3 2 }(C( R) )}
\leq 
C\|\Delta \psi_R  \|_{L^{\infty}  (C (R) )
}\|1\|_{L^{\frac 3 2 }(C( R) )}.
\end{equation}
Now, considering  the definition of  $\psi_R$,  we get 
$$
\|\Delta \psi_R  \|_{L^{\infty}
( C(R) )
}\leq CR^{-2}$$
 and we obtain
$$\|\Delta \psi_R \|_{L^{ \frac 3 2 }
(\mathcal{C}(  R) )}\leq 
C R^{-2}\|1\|_{L^{ \frac 3 2 }(C(  R) )}
.$$
Now, by stressing the fact   that 
$|
\mathcal{C}( R) 
|=CR^3$ and $R>1$, we obtain
\begin{align}\label{EstimationC1}
\|\Delta \psi_R \|_{L^{\frac 3 2 }(C( R) )}&\leq
  C R^{-2+\frac 3{ \frac 3 2 }} =C
  ,
\end{align}
and thus, we get  
\begin{equation}\label{estimateI12}
|I_1(R)|\leq 
  C  
\|\vu\|_{L^{ 6 }( \mathcal{C}( R)  )}^2.
\end{equation}
Then, since $\|\vu\|_{L^{6}( \mathcal{C}( R)  )} \to 0$ as $R\to +\infty$, we conclude 
$
I_1(R) \to_{R\to +\infty} 0. 
$ 

 \noindent {\bf Limit for $I_2(R)$}. Note that, by considering the definition of $\psi_R$ we know that $\supp( \vn \psi_R)\subset\mathcal{C}( R) $ {and it is radial}. Thus, we can write 
{
\color{black}
\begin{eqnarray}
\left|I_2(R)\right|&=&\left|\int_{B(R)} \vn \psi_R \cdot\left[\left(\pi+\frac{|\vu|^2}{2}\right) \vu\right]dx\right|\notag\\
&=& 
\left|\int_{B(R)} 
\left(\pi+\frac{|\vu|^2}{2}\right) 
\vu_\rho
 \partial_\rho\psi_R 
 dx\right|
\\
& \leq &\frac{1}{2}  
\int_{\mathcal{C}( R) }|\vu|^2
| \vu_\rho |
| \partial_\rho\psi_R |
 dx 
+
 \int_{\mathcal{C}( R) } 
 |\pi |
 | \vu_\rho |
| \partial_\rho\psi_R |  
  dx 
=: I_{21}(R)+I_{22}(R)
.
\end{eqnarray}
}

With this at hand in the following we aim to prove 
$$
\displaystyle \lim_{R\to +\infty}I_{21}(R)= \lim_{R\to +\infty} I_{22}(R)=0.
$$
To deal with the term $I_{21}(R)$, by the H\"older inequality with 
$ 
1=\frac{2}{6}+\frac{1}{p} +\frac{1}{q},
$ 
where $p^{-1}+q^{-1}=2/3$
we write 
\begin{align}\label{forheuristic1}
I_{21}(R)= \frac{1}{2}  
\int_{\mathcal{C}( R) }|\vu|^2
| \vu_\rho |
| \partial_\rho\psi_R |
 dx 
&\leq   C
\|\vu \|_{L^{ 6 }(\mathcal{C}( R)  )}^2
\|u_\rho \|_{L^{ p }(\mathcal{C}( R) )}
\|\partial_\rho \psi_R \|_{L^{ q }(\mathcal{C}( R) )}
. 
\end{align}
Since,   $ \| \partial_\rho \psi_R  \|_{L^{\infty}}\leq CR^{-1}$ and $R>1$,  we obtain  
\begin{align}
\| \partial_\rho \psi_R \|_{L^{q}(\mathcal{C}( R) )}
&
\leq 
C\|\partial_\rho \psi_R  \|_{L^{\infty}  }
\|1\|_{L^{ q }(C( R) )}
=
 C R^{-1+\frac 3{q}}  
 .
\end{align}
Now, considering that $q^{-1}=2/3 - p^{-1}$ and $3/2<p\leq 3$, we have $3 \leq q <+\infty$, and thus we get the uniform bound,  
\begin{equation}
\| \partial_\rho \psi_R \|_{L^{q}(\mathcal{C}( R) )} \leq  C, \quad 
\text{ for all }R> 1.
\end{equation}
By gathering the estimates together, we can write
\begin{equation}
I_{21}(R)\leq 
  C  
  \|u_\rho \|_{L^{ p }(\mathcal{C}( R) )}
\|\vu \|_{L^{6}( \mathcal{C}( R)  )}^2
\leq 
  C  
  \|u_\rho \|_{L^{ p }(\R^3 )}
\|\vu \|_{L^{6}( \mathcal{C}( R)  )}^2
. 
\end{equation}
Thus, with this information at hand and since $\|u\|_{L^{6}( \mathcal{C}( R)  )} \to 0$ as $R$ evolves to $+\infty$, we conclude $I_{21}(R) \to 0$ as ${R\to +\infty}$.
Now we analyze the term $I_{22}(R)$. Considering H\"older inequalities with $ \frac{2}{6}+\frac{1}{p } + \frac 1 q =1$, where $p^{-1}+q^{-1}=2/3$,  by arguing in the same manner than before, we  get the estimates 
\begin{align}\label{forheuristic2}
I_{22}(R)
= 
\int_{\mathcal{C}( R) } 
 |\pi |
 | \vu_\rho |
| \partial_\rho\psi_R |  
  dx 
&\leq 
C 
\|\pi \|_{L^{\frac{6}{2}}(C( R) )}
\|  \partial_\rho\psi_R \|_{L^{q }(C( R) )} 
\| \vu_\rho \|_{L^{p}( C ( R) )} 
\\
&\leq 
 C R^{-1+\frac 3{q }}
\|\pi \|_{L^{\frac{6 }{2}}( C( R) )}\| \vu_\rho \|_{L^{p}( C( R) )}
\\
&=
 C 
\|\pi \|_{L^{\frac{6 }{2}}( \R^3  )}\| \vu_\rho \|_{L^{p}(  C( R) )}
.
\end{align}
%
Now,  by using the divergence-free property of $\vu$, we get the following identity for the pressure term
$ \displaystyle{\pi=\sum_{i,j=1}^3\mathcal{R}_i\mathcal{R}_i(u_iu_j)}
,
$ 
where $\mathcal{R}_i$ stands for the usual Riesz transforms.
Then, gathering this relationship with the hypothesis $\vu\in L^{6}(\R^3 )$ and the fact that 
 the Riesz transform are bounded in $L^{r}$ spaces for $r\in (1,\infty)$,  we  conclude   
  \begin{equation}
  \|\pi \|_{L^{ 3}( \R^3 )} \leq 
  C 
    \|u \|^2_{L^{ 6}( \R^3 )  } <+\infty
 ,
  \end{equation}
and then, we can write  
    \begin{align}
I_{22}(R)\leq 
 C  
 \|\vu \|^2_{L^{6}( \R^3 )}
\| \vu_\rho \|_{L^{p}( \R^3 )}
.
\end{align}
Then, considering that   
$\|\vu_\rho\|_{L^{p}( \mathcal{C}( R)   )}\to 0$ as  $R\to +\infty$, 
we get  $ \lim_{R\to +\infty}
I_{22}(R) 
=0$.  Hence, by gathering the limits for $I_{21}(R)$ and $I_{22}(R)$, we deduce the desired limit for $I_{2}(R) $.\\

Thus, by mixing
  the limits for $I_{1}(R)$ and $I_{2}(R)$ with the estimate \eqref{ineqmother}, we conclude
\begin{equation}\label{LimiteSobolev}
\lim_{R\to +\infty}\int_{B(R/2)}|\vn \otimes \vu|^2 dx=\|\vu\|_{\dot{H}^1}=0.
\end{equation}
Then, by considering Sobolev embeddings, we get $\|\vu\|_{L^6}=0$, and in consequence  $\vu= 0$. 
\end{proof}


\paragraph{\bf Datasets} Is important to mention that data sharing does not apply to this article as no datasets were generated or analyzed during the current research study.
 
\paragraph{\bf  Conflict of interest} The author declares no conflict of interest and confirms being the sole contributor to this paper.



\begin{thebibliography}{99}
%

\bibitem{Chae23}
{\sc D.~Chae}, {\em Anisotropic Liouville type theorem for the stationary Navier–Stokes
equations in $\R^3$}, 
Appl. Math. Lett. 142, Paper No. 108655, 5 pp. (2023).
%
\bibitem{ChaeWolf}
{\sc D.~Chae and J.~Wolf}, {\em On {L}iouville type theorems for the steady
  {N}avier-{S}tokes equations in {$\mathbb{R}^3$}}, J. Differential Equations,
  261 (2016), pp.~5541--5560.
%
\bibitem{CRUZ}
{\sc D.~V. Cruz-Uribe and A.~Fiorenza}, {\em Variable {L}ebesgue spaces:
Foundations and harmonic analysis}, Springer Science \& Business Media, 2013.
%
\bibitem{Diening_Libro}
{\sc L.~Diening, P.~Harjulehto, P.~H{\"a}st{\"o} and M.~Ruzicka}, {\em
Lebesgue and Sobolev spaces with variable exponents}, Springer, 2011.
%
\bibitem{galdi2011introduction}
{\sc G.~Galdi}, {\em An introduction to the mathematical theory of the
{N}avier-{S}tokes equations: Steady-state problems}, Springer Science \&
Business Media, 2011.
%
\bibitem{GW78}
{\sc D. Gilbarg and  H.F.  Weinberger}, {\em Asymptotic properties of steady plane solutions
of the Navier-Stokes equations with bounded Dirichlet integral}, Ann. Scuola
Norm. Super. Pisa. CI. Sci. 5(4)  (1978), 381–404.
%
\bibitem{lemarie2016navier}
{\sc P.~G. Lemari{\'e}-Rieusset},
{\em The {N}avier-{S}tokes problem in the 21st century}, CRC press, 2016.
%
%
\bibitem{Lerner26}
{\sc N.~
Lerner},   {\em Wiener Algebras Methods for Liouville Theorems on the Stationary Navier-Stokes System}, arXiv preprint arXiv:2601.13916  (2026).
%
\bibitem{Ser2016}  
\textsc{G. Seregin}, \emph{A Liouville type theorem for steady-state Navier-Stokes equations}.  
J. {\'E}.D.P., Expos{\'e} no IX, (2016).
%
%
\bibitem{ZZ25}
 {\sc  Z. Zhang and  Q. Zu}, 
  {\em Two improved anisotropic Liouville type theorems for the stationary 3D Navier–Stokes equations}, 
Arch. Math. 124 (2025), 571–582.

\end{thebibliography}
 \end{document}